\documentclass[11pt]{article}
\usepackage{mathrsfs}
\usepackage{amsmath,amsfonts,amssymb,rotating,amsthm}
\usepackage{hyperref}
\usepackage{array}
\usepackage{color}
\def\qed{\nopagebreak\hfill{\rule{4pt}{7pt}}}

\newtheorem{theo}{Theorem}[section]

\newtheorem{lem}[theo]{Lemma}

\newdimen\Squaresize \Squaresize=11pt
\newdimen\Thickness \Thickness=0.7pt
\def\Square#1{\hbox{\vrule width \Thickness
   \vbox to \Squaresize{\hrule height \Thickness\vss
    \hbox to \Squaresize{\hss#1\hss}
   \vss\hrule height\Thickness}
\unskip\vrule width \Thickness} \kern-\Thickness}

\def\Vsquare#1{\vbox{\Square{$#1$}}\kern-\Thickness}

\def\moins{\raise 1pt\hbox{{$\scriptstyle -$}}}

\begin{document}

\begin{center}
{\large \bf On some divisibility properties of binomial sums}
\end{center}

\begin{center}
Brian Yi Sun\\[6pt]
Department of Mathematics and System Science,\\
Xinjiang University, Urumqi, Xinjiang 830046, P. R. China\\[6pt]
Email: {\tt brianys1984@126.com}
\end{center}
\vspace{0.3cm} \noindent{\bf Abstract.} In this paper, we consider two particular binomial sums
\begin{align*}
\sum_{k=0}^{n-1}(20k^2+8k+1){\binom{2k}{k}}^5 (-4096)^{n-k-1}
\end{align*}
and
\begin{align*}
\sum_{k=0}^{n-1}(120k^2+34k+3){\binom{2k}{k}}^4\binom{4k}{2k} 65536^{n-k-1},
\end{align*}
which are inspired by two series for $\frac{1}{\pi^2}$ obtained by Guillera.
We consider their divisibility properties and prove that they are divisible by $2n^2 \binom{2n}{n}^2$ for all integer $n\geq 2$.
These divisibility properties are stronger than those divisibility results found by He, who proved the above two sums are divisible by $2n \binom{2n}{n}$ with  the WZ-method.



\section{Introduction}
In recent years, many people studied some divisibility properties of binomial coefficients. For example, Guo \cite{guo} proved some divisibility properties of binomial coefficients such as

\begin{align*}
&(2n+3)\binom{2n}{n}\biggm| 3\binom{6n}{3n}\binom{3n}{n},\\
&(10n+3)\binom{3n}{n}\biggm| 21\binom{15n}{5n}\binom{5n}{n}.
\end{align*}
Z.-W. Sun \cite{sun2,sun3} found the following divisibility properties involving binomial coefficients,
\begin{align*}
&2(2n+1)\binom{2n}{n}\biggm| \binom{6n}{3n}\binom{3n}{n},\\
&4(2n+1)\binom{2n}{n}\biggm|  \sum_{k=0}^n(20k+3){\binom{2k}{k}}^2\binom{4k}{2k}(-1024)^{n-k}.
\end{align*}
Moreover, some other divisibility properties on central binomial coefficients were investigated by many authors.  The reader can be referred to
\cite{calkin,guo,guo1,guo2} for more information and references therein. For some super-congruences on binomial sums, see, for example, \cite{chenxiehe,GZ,sun1,sun4,zudilin}.

Recently, studying on divisibility properties of binomial sums has aroused great interest of many scholars since  Z.-W. Sun proposed several conjectures on divisibilities involving central binomial coefficients in \cite{sun1}. Though these conjectures have been proved by He \cite{he1,he2} and  Mao and Zhang \cite{maozhang}, we list them in the following so that the interested reader can read them conveniently.  In fact, it is worthy noting that the binomial sums considered in this paper are inspired by two series for $\frac{1}{\pi^2}$ obtained by Guillera, it seems that there are similar divisibility properties related to every formula in the family introduced in \cite{gui03}.
\begin{theo}[{\cite[Conjecture 5.1(i)]{sun1}}]\label{theorem_divs}
For all integer $n>1$,
\begin{align*}
&2n\binom{2n}{n}\biggm|\sum_{k=0}^{n-1}(3k+1){\binom{2k}{k}}^3(-8)^{n-k-1},\\
&2n\binom{2n}{n}\biggm|\sum_{k=0}^{n-1}(3k+1){\binom{2k}{k}}^3 16^{n-k-1},\\
&2n\binom{2n}{n}\biggm|\sum_{k=0}^{n-1}(6k+1){\binom{2k}{k}}^3 256^{n-k-1},\\
&2n\binom{2n}{n}\biggm|\sum_{k=0}^{n-1}(6k+1){\binom{2k}{k}}^3 (-512)^{n-k-1},\\
&2n\binom{2n}{n}\biggm|\sum_{k=0}^{n-1}(42k+5){\binom{2k}{k}}^3 4096^{n-k-1}.
\end{align*}
\end{theo}

Besides, He \cite{he1,he2} also found two similar divisibilities involving binomial sums by using WZ-method,
\begin{theo}[{\cite[Theorem 1.1(1.4), Theorem 1.2]{he1,he2}}]\label{he's results}
For $n\geq2$, we have
\begin{align*}
&2n\binom{2n}{n}\biggm|\sum_{k=0}^{n-1}(20k^2+8k+1){\binom{2k}{k}}^5 (-4096)^{n-k-1},\\
&2n\binom{2n}{n}\biggm|\sum_{k=0}^{n-1}(120k^2+34k+3){\binom{2k}{k}}^4\binom{4k}{2k} 65536^{n-k-1}.
\end{align*}
\end{theo}

In this paper, we in fact prove the following stronger results than Theorem \ref{he's results}.

\begin{theo}\label{my-theorem-1} For $n\geq 2$, we have
\begin{align}\label{my-div-1}
2n^2\binom{2n}{n}^2\biggm|\sum_{k=0}^{n-1}(20k^2+8k+1){\binom{2k}{k}}^5 (-4096)^{n-k-1}.
\end{align}
\end{theo}
\begin{theo}\label{my-theorem-2}
For $n\geq 2$, we have
\begin{align}\label{my-div-2}
2n^2\binom{2n}{n}^2\biggm|\sum_{k=0}^{n-1}(120k^2+34k+3){\binom{2k}{k}}^4\binom{4k}{2k} 65536^{n-k-1}.
\end{align}
\end{theo}

The present paper is organized as follows. We recall the definition of WZ-pair and also make some preliminaries in Section \ref{sec-2}. In section \ref{sec-3}, we give the proofs of Theorem \ref{my-theorem-1} and Theorem \ref{my-theorem-2}.

\section{Preliminaries}\label{sec-2}
We recall that a function $A(n,k)$ is hypergeometric in its two variables if the quotients
\begin{equation*}
\frac{A(n+1,k)}{A(n,k)} \text{~and~} \frac{A(n,k+1)}{A(n,k)}
\end{equation*}
are rational functions in $n$ and $k$, respectively.
We say that a pair of hypergeometric functions  $F(n,k)$ and $G(n,k)$ form a WZ-pair if $F$ and $G$ satisfy the following relation:
\begin{align}\label{wzpair}
F(n,k-1)-F(n,k)=G(n+1,k)-G(n,k).
\end{align}
Wilf and Zeilberger \cite{wz} proved in this case that there exists a rational function $C(n,k)$ such that
$$F(n,k)=C(n,k)G(n,k).$$
The rational function $C(n,k)$ is also called the certificate of the pair $(F,G)$. One can use a Maple package written by Zeilberger to discover or verify WZ-pairs, see EKHAD \cite[Appendix A]{a=b}.

To prove our results, we need the following lemma, which is due to He.
\begin{lem}[{\cite[Lemma 2.3]{he1}}]\label{my-lemma-1}
Let $n\geq 0, k\geq 1 $ and $N\geq 1$ be integers and $P(N)$  a positive integer depending on $N$. Assume that  $F(n,k)$ and $G(n,k)$ are two rational functions in $n,k$ such that $F(n,k)=G(n,k)=0$ for $k>n$ and they form a WZ-pair.
If we have
\begin{itemize}
\item[(i)]
$P(N)\biggm|B^N \sum_{k=1}^{N-1}G(N,k)
$ for all integers  $k\geq 0$;
and
\item[(ii)] $P(N)\Bigm|B^NF(N-1,N-1),$
\end{itemize}
where $B$ is an integer such that $B^NF(N-1,N-1)$ and $B^N G(N,k)$ are both integers.
Then we have
\begin{equation*}
P(N)\Bigm|B^N\sum_{n=0}^{N-1}F(n,0).
\end{equation*}
\end{lem}

Let $p$ be a prime. The $p$-adic evaluation of an integer $m$ is defined by
$$v_p(m)=sup\{a\in \mathbb{N}:p^a|m\}.$$
 Clearly, for a rational number $x=m/n$ with $m,n\in \mathbb{Z}$, we  can define
  \begin{equation}\label{p-adic-ratio}v_p(x)=v_p(m)-v_p(n).
  \end{equation}
   For the $p$-adic evaluation of $n!$, it is well known that
 \begin{equation}\label{p-adic-factorial}
 v_p(n!)=\sum_{i=1}^\infty\left\lfloor\frac{n}{p^i}\right\rfloor,
 \end{equation}
 where $\lfloor x\rfloor$ is the greatest integer not exceeding $x$.

We also need the following lemmas before giving our proofs. Let $\mathbb{N}$ and $\mathbb{Z}$ represents, respectively, the set of nonnegative integers and integers throughout the paper.
 \begin{lem}[{\cite[Lemma 2.2]{he1}}]\label{lem:div6-3} Let $n$ be a positive integer and $k\leq n$ a nonnegative integer. Then
\begin{equation*}
(2n+2k-1)\binom{2k}{k}\Biggm|n \binom{2 n}{n} \binom{ 2n+2k}{n+k}\binom{n+k}{2k}.
\end{equation*}
\end{lem}
 \begin{lem}\label{lem:div6-3'} Let $n\geq 2$ be a positive integer. Then
\begin{equation*}
64(2n+1)\biggm| n^2 (n+1) \binom{2 n}{n} \binom{2 n-2}{n-1} \binom{2 n+2}{n+1}.
\end{equation*}
\end{lem}
\begin{proof}
In fact, we note that
\begin{align*}
\frac{ n^2 (n+1) \binom{2 n}{n} \binom{2 n-2}{n-1} \binom{2 n+2}{n+1}}{64(2n+1)}&=\frac{(2n-1)^2(2n-3)!^3}{(n-2)!^3(n-1)!^3}\\
&=(2n-1)^2\binom{2n-3}{n-1}^3.
\end{align*}
So this follows the statement.
\end{proof}
\qed
\begin{lem}\label{lem:2-div-2-1} Let $m>1$  and $n, k$ be nonnegative integers. Then for $n\geq k$, we have
\begin{equation*}
\begin{split}
\left\lfloor\frac{4n+2k-2}{m}\right\rfloor+3\left\lfloor\frac{k}{m}\right\rfloor+\left\lfloor\frac{2n}{m}\right\rfloor&\geq 3\left\lfloor\frac{2k}{m}\right\rfloor+\left\lfloor\frac{n}{m}\right\rfloor+\left\lfloor\frac{n-1}{m}\right\rfloor+2\left\lfloor\frac{n-k}{m}\right\rfloor
\\
&\quad+\left\lfloor\frac{2n+k-1}{m}\right\rfloor.
\end{split}
\end{equation*}
\end{lem}
\begin{proof}
Let
\begin{align*}
S&=\left\lfloor\frac{4n+2k-2}{m}\right\rfloor+3\left\lfloor\frac{k}{m}\right\rfloor+\left\lfloor\frac{2n}{m}\right\rfloor- 3\left\lfloor\frac{2k}{m}\right\rfloor-\left\lfloor\frac{n}{m}\right\rfloor-\left\lfloor\frac{n-1}{m}\right\rfloor
\\
&\quad-2\left\lfloor\frac{n-k}{m}\right\rfloor-\left\lfloor\frac{2n+k-1}{m}\right\rfloor.
\end{align*}
Then
\begin{align*}
S&=3\left\{\frac{2k}{m}\right\}+\left\{\frac{n}{m}\right\}+\left\{\frac{n-1}{m}\right\}+2\left\{\frac{n-k}{m}\right\}
+\left\{\frac{2n+k-1}{m}\right\}-\left\{\frac{4n+2k-2}{m}\right\}\\
&\quad-3\left\{\frac{k}{m}\right\}
-\left\{\frac{2n}{m}\right\},
\end{align*}
 since $(6k)+(n-1)+n+2(n-k)+(2n+k-1)-(4n+2k-2)-3k-(2n)=0$. To prove $S\geq 0$ is equivalent to show
\begin{equation}\label{ineq-2}
\begin{split}
\left\{\frac{4n+2k-2}{m}\right\}+3\left\{\frac{k}{m}\right\}+\left\{\frac{2n}{m}\right\}&\leq 3\left\{\frac{2k}{m}\right\}+\left\{\frac{n}{m}\right\}+\left\{\frac{n-1}{m}\right\}\\
&\quad+2\left\{\frac{n-k}{m}\right\}
+\left\{\frac{2n+k-1}{m}\right\}
\end{split}
\end{equation}
Note that this only depends on $n$ and $k$ modulo $m$. Without loss of generality, we may assume that $n,k\in\{0,1,\ldots,m-1\}$.
So
$$0\leq \frac{k}{m}<1, 0\leq \frac{n}{m}<1,  \frac{n-1}{m}<1 \text{~and~}-1<\frac{n-k}{m}<1. $$ We can simplify the inequality \eqref{ineq-2} as follows:
\begin{align}\label{ineq-2'}
3\left\{\frac{2k}{m}\right\}+\frac{4n-5k-1}{m}
+\left\{\frac{2n+k-1}{m}\right\}
\geq\left\{\frac{4n+2k-2}{m}\right\}
+\left\{\frac{2n}{m}\right\}.
\end{align}

It is not difficult to see that the inequality \eqref{ineq-2'} holds for $k=0$.

Let us consider the following cases and note that $n\geq k$:
\begin{itemize}
\item[1)]$\frac{n}{m}<\frac{1}{2}$,
\begin{itemize}
\item[a)]if $2n+k-1<m$, in this case the inequality \eqref{ineq-2'} reduces to
\begin{align*}
\frac{6k}{m}+\frac{4n-5k-1}{m}
+\frac{2n+k-1}{m}-\frac{2n}{m}
\geq\left\{\frac{4n+2k-2}{m}\right\}
,
\end{align*}
which holds trivially since
$$\frac{4n+2k-2}{m}\geq\left\{\frac{4n+2k-2}{m}\right\}.$$

\item[b)]if $m\leq 2n+k-1<2m$, in this case the inequality \eqref{ineq-2'} reduces to
\begin{align*}
\frac{6k}{m}+\frac{4n-5k-1}{m}
+\frac{2n+k-1}{m}-1
\geq\left\{\frac{4n+2k-2}{m}\right\}
+\frac{2n}{m},
\end{align*}
which holds  since $(4n+2k-2)/m\geq 2$ and thereby
$$\frac{4n+2k-2}{m}-1\geq1> \left\{\frac{4n+2k-2}{m}\right\}.$$
\end{itemize}
\item[2)]$\frac{n}{m}\geq \frac{1}{2}$ and $ \frac{k}{m}<\frac{1}{2}$.
\begin{itemize}
\item[a)]if $m\leq 2n+k-1<2m$, then the inequality \eqref{ineq-2'} reduces to
\begin{align*}
\frac{6k}{m}+\frac{4n-5k-1}{m}
+\frac{2n+k-1}{m}-1
\geq\left\{\frac{4n+2k-2}{m}\right\}
+\frac{2n}{m}-1,
\end{align*}
which is equivalent to the trivial inequality
$$\frac{4n+2k-2}{m}\geq\left\{\frac{4n+2k-2}{m}\right\}.$$
\item[b)]if $2m\leq 2n+k-1<\frac{5m}{2}$, then in this case the inequality \eqref{ineq-2'} reduces to
\begin{align*}
\frac{6k}{m}+\frac{4n-5k-1}{m}
+\frac{2n+k-1}{m}-2
\geq\left\{\frac{4n+2k-2}{m}\right\}
+\frac{2n}{m}-1,
\end{align*}
which is equivalent to the trivial inequality
$$\frac{4n+2k-2}{m}-1\geq\left\{\frac{4n+2k-2}{m}\right\},$$
which holds clearly since
$$\left\{\frac{4n+2k-2}{m}\right\}=\frac{4n+2k-2}{m}-2.$$
\end{itemize}
\item[(3)]$\frac{k}{m}\geq \frac{1}{2}$,
\begin{itemize}
\item[a)] if $\frac{3m}{2}\leq 2n+k-1<2m$, then in this case the inequality \eqref{ineq-2'} reduces to
\begin{align*}
\frac{6k}{m}-3+\frac{4n-5k-1}{m}
+\frac{2n+k-1}{m}-1
\geq\left\{\frac{4n+2k-2}{m}\right\}
+\frac{2n}{m}-1,
\end{align*}
which is equivalent to the following inequality
\begin{equation}\label{div-2-ineq}
\frac{4n+2k-2}{m}-3\geq\left\{\frac{4n+2k-2}{m}\right\}.
\end{equation}
The inequality \eqref{div-2-ineq} follows from
$$\left\{\frac{4n+2k-2}{m}\right\}=\frac{4n+2k-2}{m}-3.$$
\item[b)]if $2m\leq 2n+k-1<3m$, then the inequality \eqref{ineq-2'} reduces to
\begin{align*}
\frac{6k}{m}-3+\frac{4n-5k-1}{m}
+\frac{2n+k-1}{m}-2
\geq\left\{\frac{4n+2k-2}{m}\right\}
+\frac{2n}{m}-1,
\end{align*}
which is equivalent to the following inequality
\begin{equation}\label{div-2-ineq'}
\frac{4n+2k-2}{m}-4\geq\left\{\frac{4n+2k-2}{m}\right\}.
\end{equation}
The inequality \eqref{div-2-ineq'} follows from
\begin{align*}
\left\{\frac{4n+2k-2}{m}\right\}=
\begin{cases}
 \frac{4n+2k-2}{m}-4,
& \text{if}~2m\leq 2n+k-1<\frac{5m}{2};\\
 \frac{4n+2k-2}{m}-5,
&  \text{if}~\frac{5m}{2}\leq 2n+k-1<3m.
\end{cases}
\end{align*}
\end{itemize}
\end{itemize}

This completes the proof.
\end{proof}
\qed

\begin{lem}\label{lem:2-div_2-2} Let $n$ be a positive integer and $k$ a nonnegative integer. Then
\begin{equation*}
{\binom{2 k}{k}^2}\Biggm|\binom{2 n}{n} \binom{n}{k} \binom{k+n}{2 k} \binom{k+2 n-1}{n-1} \binom{2 k+4n-2}{k+2 n-1}.
\end{equation*}
\end{lem}
\begin{proof}
In view of \eqref{p-adic-factorial},
since
\begin{align*}
W(n,k)=\frac{\binom{2 n}{n} \binom{n}{k} \binom{k+n}{2 k} \binom{k+2 n}{n} \binom{2 k+4 n}{k+2 n}}{\binom{2 k}{k}^2}=\frac{k!^3 (2 n)! (2 k+4 n-2)!}{(2 k)!^3 n!(n-1)! (n-k)!^2 (k+2 n-1)!},
\end{align*}
it suffices to show that for any prime $p$, $v_p(W(n,k))=\sum_{i=1}^\infty A_{p^i}\geq 0$ where
\begin{align*}A_{p^i}&=\left\lfloor\frac{4n+2k-2}{p^i}\right\rfloor+3\left\lfloor\frac{k}{p^i}\right\rfloor+\left\lfloor\frac{2n}{p^i}\right\rfloor- \left(3\left\lfloor\frac{2k}{p^i}\right\rfloor+\left\lfloor\frac{n}{p^i}\right\rfloor+\left\lfloor\frac{n-1}{p^i}\right\rfloor+2\left\lfloor\frac{n-k}{p^i}\right\rfloor
\right.\\
&\quad\left.+\left\lfloor\frac{2n+k-1}{m}\right\rfloor\right). \end{align*}
By Lemma \ref{lem:2-div-2-1}, it follows that $A_{p^i}\geq 0$ and thereby $v_p(W(n,k))\geq 0$. Hence, $W(n,k)\in \mathbb{Z}$, which proves
Lemma \ref{lem:2-div_2-2}.
\end{proof}
\qed

\begin{lem}\label{lem:1-div}
 For positive integer $n\geq1$, we have
\begin{align*}
((2 n-1)! (2n-2)!) (3 n-3)!\bigm|(6 n-5)! (n-1)!.
\end{align*}
\end{lem}
\begin{proof}
Along the same line of the proof of Lemma
\ref{lem:2-div_2-2}, it suffices to show that for any $m\geq 2$,
\begin{align}\label{inequality-lem-1}
\left\lfloor\frac{6n-5}{m}\right\rfloor+\left\lfloor\frac{n-1}{m}\right\rfloor\geq \left\lfloor\frac{2n-1}{m}\right\rfloor
+\left\lfloor\frac{2n-2}{m}\right\rfloor+\left\lfloor\frac{3n-3}{m}\right\rfloor,
\end{align}
which is equivalent to the following inequality
\begin{align}\label{inequality-lem}
\left\{\frac{6n-5}{m}\right\}+\left\{\frac{n-1}{m}\right\}\leq \left\{\frac{2n-1}{m}\right\}
+\left\{\frac{2n-2}{m}\right\}+\left\{\frac{3n-3}{m}\right\}
\end{align}
since $(6n-5)+(n-1)-(2n-1)-(2n-2)-(3n-3)=0$.

Note that $n\geq 1$, we may assume that $n\in\{1,\ldots,m-1,m\}$ since it only depends on $m$. To begin with, it is easy to verify that the inequality \eqref{inequality-lem}
holds for $n=1$ and $n=m$. In the following, we assume that $2\leq n\leq m-1.$ So we have $3n-3\geq 2n-1>2n-2$.  To prove \eqref{inequality-lem}, it is sufficient to consider
the following cases.
\begin{itemize}
\item[1)]$3n-3<m$. In this case, the inequality \eqref{inequality-lem} reduces to
\begin{align*}
\left\{\frac{6n-5}{m}\right\}+\frac{n-1}{m}\leq \frac{2n-1}{m}
+\frac{2n-2}{m}+\frac{3n-3}{m},
\end{align*}
i.e.,
\begin{align*}
\left\{\frac{6n-5}{m}\right\}\leq \frac{6n-5}{m},
\end{align*}
which is clearly true.
\item[2)]$m\leq 3n-3\leq 2m-1$,
\begin{itemize}
\item[a)]if $2n-1< m$, then $2m+1\leq 6n-5\leq 4m-1$. So
the inequality \eqref{inequality-lem} reduces to
\begin{align*}
\left\{\frac{6n-5}{m}\right\}\leq \frac{6n-5}{m}-1,
\end{align*}
which is valid since
$2<2+\frac{1}{m}\leq \frac{6n-5}{m}\leq 4-\frac{1}{m}<4.$
\item[b)]if $m\leq 2n-2<2 m$, then $3m+1\leq 6n-5< 6m+1$. So
the inequality \eqref{inequality-lem} reduces to
\begin{align*}
\left\{\frac{6n-5}{m}\right\}\leq \frac{6n-5}{m}-3,
\end{align*}
which is valid since
$3<3+\frac{1}{m}\leq \frac{6n-5}{m}< 6+\frac{1}{m}<6.$

\item[c)] if $2n-2=m-1$, then $n=(m+1)/2$. So the inequality \eqref{inequality-lem} reduces to
\begin{align*}
\left\{\frac{3m-2}{m}\right\}+\frac{m-1}{2m}\leq
\frac{m-1}{m}+\frac{3m-3}{2m},
\end{align*}
i.e.,
\begin{align*}
\left\{\frac{3m-2}{m}\right\}\leq
\frac{3m-2}{m}-1.
\end{align*}
It holds trivially since $2\leq \frac{3m-2}{m}=3-\frac{2}{m}<3$.
\end{itemize}
\item[3)]$2m\leq 3n-3\leq 3m-6$. In this case, we have $m<4m/3\leq 2n-2\leq 2m-4$ and $4m+1\leq 6n-5\leq 6m-11$. So
the inequality \eqref{inequality-lem} reduces to
\begin{align*}
\left\{\frac{6n-5}{m}\right\}\leq \frac{6n-5}{m}-4,
\end{align*}
which is also true since
$4<4+\frac{1}{m}\leq \frac{6n-5}{m}\leq 6-\frac{11}{m}<6.$
\end{itemize}
On the basis of above discussion, it follows the inequality \eqref{inequality-lem} and thus the inequality \eqref{inequality-lem-1}
is true. Therefore, by invoking \eqref{p-adic-ratio} and \eqref{p-adic-factorial}, we can obtain our statement.

This completes the proof.
\end{proof}
\qed

Now we are ready to prove our main results.
\section{Proofs of Theorem \ref{my-theorem-1} and  Theorem \ref{my-theorem-2}}\label{sec-3}
\noindent
 \emph{Proof of Theorem \ref{my-theorem-1}.}
  For $n,k\in \mathbb{N}$, let
\begin{align*}
F(n,k)=\frac{ (-1)^{n+k} \left(20 n^2-12 k n-2 k+8 n+1\right)\binom{2 n}{n}^3 \binom{2 k+2 n}{k+n} \binom{2 n-2 k}{n-k} \binom{k+n}{n-k}}{\binom{2 k}{k} 4^{6 n-2 k}}
\end{align*}
and
\begin{align*}
G(n,k)=\frac{(-1)^{n+k}2 n^3 \binom{2 n}{n}^3 \binom{2 k+2 n}{k+n} \binom{2 n-2 k}{n-k} \binom{k+n}{n-k}}{\binom{2 k}{k} (2 k+2 n-1) 16^{-k+3 n-1}}.
\end{align*}
The numbers $F(n,k)$ can be found in \cite{he1} and \cite{zudilin} and the numbers $G(n,k)$ can be obtained by the Gosper algorithm for
$F(n,k-1)-F(n,k)$; see \cite{gosper}. One can easily deduce that the pair $(F,G)$ forms a WZ-pair. They satisfy the relation
\eqref{wzpair}.

By Lemma \ref{lem:div6-3'}, for $k=1$ and $N\geq 2$, we have
\begin{align*}
\frac{(-4096)^{N-1} G(N,1)}{2N^2\binom{2N}{N}^2}=\frac{ N^2 (N+1) \binom{2 N}{N} \binom{2 N-2}{N-1} \binom{2 N+2}{N+1}}{64 (2 N+1)}\in \mathbb{Z}.
\end{align*}

 In view of Lemma \ref{lem:div6-3}, we obtain that for $N\geq k\geq 2$,
\begin{align*}
\frac{(-4096)^{N-1} G(N,k)}{2N^2\binom{2N}{N}^2}=\frac{2^{4 k-8} N \binom{2 N}{N} (-1)^{k+1} \binom{k+N}{N-k} \binom{2 N-2 k}{N-k} \binom{2 k+2 N}{k+N}}{\binom{2 k}{k} (2 k+2 N-1)}\in \mathbb{Z}.
\end{align*}
Hence it follows that
\begin{align*}
2N^2\binom{2N}{N}^2\biggm|(-4096)^{N-1} \sum_{k=1}^{N-1}G(N,k).
\end{align*}

Furthermore, for $N\geq 2$, we have
\begin{align*}
\frac{(-4096)^{N-1} F(N-1,N-1)}{2N^2\binom{2N}{N}^2}&=\frac{(-1)^{N+1} 2^{4 N-5} \left(8 N^2-10 N+3\right) \binom{2 N-2}{N-1}^2 \binom{4 N-4}{2 N-2}}{N^2 \binom{2 N}{N}^2}\\
&=\frac{(-1)^{N+1} 2^{4 N-7} (4 N-3) \binom{4 N-4}{2 N-2}}{2 N-1}\in \mathbb{Z}
\end{align*}
since
\begin{align*}
\frac{\binom{4 N-4}{2 N-2}}{2 N-1}=\binom{4N-4}{2N-2}-\binom{4N-4}{2N-3}.
\end{align*}
From Lemma \ref{my-lemma-1}, it follows that
\begin{align*}
2N^2\binom{2N}{N}^2\biggm|(-4096)^{N-1} \sum_{n=0}^{N-1}F(n,0)=\sum_{n=0}^{N-1}\left(20 n^2+8 n+1\right) \binom{2 n}{n}^5 (-4096)^{ N- n-1},
\end{align*}
which is nothing but the divisibility \eqref{my-div-1}.

\qed

\noindent
\emph{Proof of Theorem \ref{my-theorem-2}.}
Let
\begin{align*}
F(n,k)=\frac{ 16^{k-4 n} \left(120 n^2-84 k n-10 k+34 n+3\right) \binom{2 n}{n}^3\binom{2 k+4 n}{k+2 n} \binom{k+2 n}{n} \binom{k+n}{2 k} \binom{n}{k}}{\binom{2 k}{k}^2}
\end{align*}
and
\begin{align*}
G(n,k)=\frac{n^2 2^{4 k-16 n+10} \binom{2 n}{n}^3 \binom{n}{k} \binom{k+n}{n-k} \binom{k+2 n-1}{n-1} \binom{2 k+4 n-2}{k+2 n-1}}{\binom{2 k}{k}^2}.
\end{align*}
The numbers $F(n,k)$ can be found in \cite{gui,he2,zudilin} and the numbers $G(n,k)$ can be obtained by the Gosper algorithm for
$F(n,k-1)-F(n,k)$; see \cite{gosper}. It is easy to verify that this pair $(F,G)$ satisfies the relation \eqref{wzpair}.

We first notice that for all nonnegative integers $N\geq2$ and $N\geq k$,
\begin{align}\label{div-2-Integer-2}
\frac{2^{16N-16}G(N,k)}{2N^2\binom{2N}{N}^2}&=\frac{2^{4 k-7} \binom{2 N}{N} \binom{N}{k} \binom{k+N}{N-k} \binom{k+2 N-1}{N-1} \binom{2 k+4 N-2}{k+2 N-1}}{\binom{2 k}{k}^2}\in \mathbb{Z}
\end{align}
and
\begin{align*}
&\frac{2^{16N-16}F(N-1,N-1)}{2N^2\binom{2N}{N}^2}\\
&\quad=\frac{3\cdot 2^{4 N-5} \left(12 N^2-16 N+5\right) \binom{2 N-2}{N-1} \binom{3 N-3}{N-1} \binom{6 N-6}{3 N-3}}{N^2 \binom{2 N}{N}^2}\\
&=\frac{3\cdot 2^{4N-7}(6N-5)!(N-1)!}{(2N-1)!(2N-2)!(3N-3)!}\in \mathbb{Z},
\end{align*}
in view of Lemma \ref{lem:2-div_2-2} and Lemma \ref{lem:1-div}, respectively.

Thus, it follows that
\begin{align*}
2N^2\binom{2N}{N}^2\biggm|2^{16N-16} \sum_{k=1}^{N-1}G(N,k).
\end{align*}
Therefore, using Lemma \ref{my-lemma-1}, we arrive at
\begin{align}\label{div-2}
2N^2\binom{2N}{N}^2\biggm|2^{16N-16} \sum_{n=0}^{N-1}F(n,0).
\end{align}
Since
\begin{align*}
2^{16N-16} \sum_{n=0}^{N-1}F(n,0)=\sum_{n=0}^{N-1}\left(120 n^2+34 n+3\right) \binom{2 n}{n}^4 \binom{4 n}{2 n} 65536^{N-n-1},
\end{align*}
 we obtain the divisibility \eqref{my-div-2} by \eqref{div-2}.
This finishes the proof of Theorem \ref{my-theorem-2}.
\qed

\vspace{.3cm}

\noindent{\bf Acknowledgments.} The author would like to thank the referee for his/her helpful comments and suggestions.
 This work was supported by the Scientific Research Program of the Higher Education Institution of
Xinjiang Uygur Autonomous Region (No.~XJEDU2016S032).


\begin{thebibliography}{99}
\bibitem{calkin}N.J. Calkin, Factors of sums of powers of binomial coefficients, Acta Arith. {\bf 86}(1998), 17--26.
\bibitem{chenxiehe}Y.G. Chen, X.Y. Xie and B. He, On some congruences of certain binomial sums, Ramanujan J. {\bf (2)40}(2016), 237--244.
\bibitem{gosper}R.W. Gosper, Decision procedure for indefinite hypergeometric summation, Proc. Natl. Acad. Sci. USA {\bf 75}(1978), 40--42.
\bibitem{gui}J. Guillera, Some binomial series obtained by the WZ-method, Adv. in Appl. Math. {\bf 29}(2002), 599--603.
 \bibitem{gui03}J. Guillera, About a new kind of Ramanujan type series, Exp. Math. {\bf 12}(2003), 507--510.
\bibitem{GZ}J. Guillera and W. Zudilin, ``Divergent" ramanujan-type supercongruences, Proc. Amer. Math.  Soc. {\bf  (3)140}(2012), 765--777.
\bibitem{guo}V.J.W. Guo, Proof of two divisibility properties of binomial coefficients conjectured by Z.-W. Sun, Electron. J. Combin.  {\bf (4)20}(2013), \#P20.
\bibitem{guo1} V.J.W. Guo and C. Krattenthaler, Some divisibility properties of binomial and $q$-binomial coefficients,
J. Number Theory {\bf 135} (2014), 167--184.
\bibitem{guo2}V.J.W. Guo, F. Jouhet and J. Zeng, Factors of alternating sums of products of binomial and $q$-binomial
coefficients, Acta Arith. {\bf 127}(2007), 17--31.
\bibitem{he1}B. He, On the divisibility properties concerning sums of binomial coefficients, Ramanujan J. doi:10.1007/s11139-016-9780-6, 2016.
\bibitem{he2}B. He, On the divisibility properties of certain binomial sums,  J. Number Theory. {\bf 147}(2015), 133--140.
\bibitem{maozhang}G.S. Mao and T. Zhang, Proof of two conjectures of Z.-W. Sun, preprint, 2016, arXiv:1511.05553v4.
\bibitem{a=b} M. Petkov\v{s}ek, H.S. Wilf and D. Zeilberger, A = B (A K Peters, Ltd., Wellesley, MA,
1996).
\bibitem{sun1} Z.-W. Sun, Super congruences and Euler numbers, Sci. China Math. {\bf 54}(2011), 2509--2535.
\bibitem{sun2} Z.-W. Sun, Products and sums divisible by central binomial coefficients, Electron. J. Combin. {\bf (1)20}(2013), \#P9.
\bibitem{sun3} Z.-W. Sun, On divisibility of binomial coefficients, J. Aust. Math. Soc. {\bf 93}(2012) 189--201.
\bibitem{sun4} Z.-W. Sun, Open conjectures on congruences, preprint, 2011, arXiv:0911.5665v59
\bibitem{wz}H.S. Wilf and D. Zeilberger, Rational functions certify combinatorial identities, J. Amer. Math. Soc. {\bf 3}(1990), 147--158.
\bibitem{zudilin} W. Zudilin, Ramanujan-type supercongruences, J. Number Theory {\bf 129}(2009), 1848--1857.
\end{thebibliography}
\end{document}